\documentclass[reqno,12pt]{amsart}

\usepackage{color}
\usepackage{enumerate}

\numberwithin{equation}{section}

\theoremstyle{plain}
\newtheorem{thm}{Theorem}[section]
\newtheorem{cor}[thm]{Corollary} 
\newtheorem{lemma}[thm]{Lemma} 
\newtheorem{prop}[thm]{Proposition}

\newtheorem{hyp}[thm]{Hypotheses}

\newtheorem{remark}[thm]{Remark}

\theoremstyle{remark}

\theoremstyle{definition}

\newtheorem{defi}[thm]{Definition}

\newtheorem{notation}[thm]{Notations}

\newcommand{\ds}{\displaystyle}

\def\question#1{\ifmmode\text{\bf \color{red} Q: #1}\else{\bf
    Q}\footnote{#1}\fi:\color{red}}

\usepackage{amscd,amssymb,comment,epic,eepic,euscript,graphics}
\newcommand\eps{\epsilon}
\newcommand\vareps{\varepsilon}
\newcommand{\la}{\lambda}

\newcommand\Acal{{\mathcal{A}}}
\newcommand\Ncal{{\mathcal{N}}}
\newcommand\At{{\bf A}}
\newcommand\Bt{{\bf B}}
\newcommand\Xt{{\bf X}}
\newcommand\Vt{{\bf V}}
\newcommand\Ut{{\bf U}}

\newcommand\R{{\mathbb R}}
\newcommand\D{{\mathbb D}}
\newcommand\C{{\mathbb C}}
\newcommand\N{{\mathbb N}}

\newcommand\T{{\mathbb T}}
\newcommand\AD{{\Acal(\D_{1+\eps}^n)}}

\newcommand{\BH}{\mathcal{B}(\mathcal{H})}
\newcommand{\Hcal}{\mathcal{H}}
\newcommand{\Ccal}{\mathcal{C}}

\newcommand\Lcal{{\mathcal{L}}}
\newcommand{\Ical}{\mathcal{I}}

\newcommand\Tr{{\mathrm{Tr}}}

\newcommand{\im}{\text{\rm Im}}
\newcommand{\re}{\text{\rm Re}}
\renewcommand{\i}{\text{\rm i}}

\begin{document}
\title{Trace formulas for tuples of commuting contractions}

\author[Skripka]{Anna Skripka$^{*}$} \email{skripka@math.unm.edu}

\thanks{\indent\llap{${}^{*}$}Research supported in part by NSF
  grant DMS-1249186.}

\address{A.S., Department of Mathematics and Statistics, University of New Mexico, 400 Yale Blvd NE, MSC01 1115, Albuquerque, NM 87131, USA}



\date{\today}

\begin{abstract}
  This paper extends the trace formulas of \cite{dsD} with perturbations in normed ideals of $\BH$ to multivariate functions of commuting contractions admitting a dilation to commuting normal contractions.
\end{abstract}

\maketitle

\section{Introduction.}

One of the fundamental results in perturbation theory is Krein's trace formula \cite{Krein}
\begin{align}
\label{Ktf}
\Tr\left(f(H_0+V)-f(H_0)\right)=\int_\R f'(\la)\xi(\la)\,d\la,
\end{align}
where $H_0$ is a self-adjoint operator defined in a separable Hilbert space $\Hcal$, $V$ a self-adjoint trace class operator, and $f$ a sufficiently nice scalar function. The integrable function $\xi$ is determined by the operators $H_0$ and $V$ and does not depend on $f$.
The formula \eqref{Ktf} has found a number of applications and generalizations (see, e.g., survey \cite{S} for details and references). In particular, when $V$ is a Hilbert-Schmidt operator, we have Koplienko's trace formula \cite{Kop}
\begin{align}
\label{Kotf}
\Tr\left(f(H_0+V)-f(H_0)-\frac{d}{dt}\bigg|_{s=0}f(H_0+sV)\right)=\int_\R f''(\la)\eta(\la)\,d\la,
\end{align}
where $\eta$ is an integrable function determined by $H_0$ and $V$.

Let $\Ical$ be a normed ideal of $\BH$ (the algebra of bounded linear operators on $\mathcal{H}$) with norm $\|\cdot\|_\Ical$ (see Definition \ref{ideal}). Let $\tau_\Ical$ be a trace on $\Ical$ bounded with respect to the ideal norm $\|\cdot\|_\Ical$. Examples of $(\Ical,\tau_\Ical)$ include $(S^1,\Tr)$, the trace class ideal with the canonical trace, and $(\Lcal^{(1,\infty)},\Tr_\omega)$, the dual Macaev ideal with the Dixmier trace $\Tr_\omega$ corresponding to a generalized limit $\omega$ on $\ell^\infty(\N)$. (Definitions and references can be found in \cite{dsD}.)

The following results were obtained in \cite{dsD}.

\begin{thm}\cite[Theorem 3.4]{dsD}
\label{dsmt1}
Let $H_0$ and $H$ be contractions. If $V=H-H_0\in\Ical$, then there exists a (countably additive, complex)
finite measure $\mu=\mu_{H_0,H}$ such that
\begin{align*}
\|\mu\|\leq \min\left\{\|\tau_\Ical\|_{\Ical^*}\cdot\|V\|_\Ical,\,
\tau_\Ical\big(|\re(V)|\big)+\tau_\Ical\big(|\im(V)|\big)\right\}
\end{align*} and, for every $f$ analytic on $\D_{1+\eps}$, a disc centered at $0$ of radius $1+\eps$, with $\eps>0$,
\begin{align}
\label{dstf2}
\tau_\Ical\big(f(H_0+V)-f(H_0)\big)=\int_\Omega f'(\la)\,d\mu(\la),
\end{align}
where $\Omega=\T$. If, in addition, $H_0$ and $V$ are self-adjoint, then there exists unique real-valued measure $\mu$ such that \eqref{dstf2} holds with $\Omega=\text{\rm conv}\,(\sigma(H_0)\cup\sigma(H_0+V))$ for $f$ real-analytic on $\Omega$.
\end{thm}

\begin{thm}\cite[Theorem 3.4]{dsD}
\label{dsmt2}
Let $H_0$ and $H$ be contractions. Suppose that $\Ical^{1/2}$ is a normed ideal with the norm satisfying \eqref{normCS}. If $V=H-H_0\in\Ical^{1/2}$, then there exists a (countably additive, complex)
finite measure $\nu=\nu_{H_0,H}$ such that
\begin{align*}
\|\nu\|\leq \frac12\tau_\Ical(|V|^2)
\end{align*} and, for every $f$ analytic on $\D_{1+\eps}$, with $\eps>0$,
\begin{align}
\label{dstf2}
\tau_\Ical\left(f(H_0+V)-f(H_0)-\frac{d}{ds}\bigg|_{s=0}f(H_0+sV)\right)=\int_\Omega f''(\la)\,d\nu(\la),
\end{align}
where $\Omega=\T$. If, in addition, $H_0$ and $V$ are self-adjoint, then there exists unique real-valued measure $\nu$ such that \eqref{dstf2} holds with $\Omega=\text{\rm conv}\,(\sigma(H_0)\cup\sigma(H_0+V))$ for $f$ real-analytic on $\Omega$.
\end{thm}

We extend the results of Theorems \ref{dsmt1} and \ref{dsmt2} to multivariate functions of commuting contractions admitting a dilation to commuting normal contractions. If $\tau_\Ical$ has a nontrivial normal component, then we also request that the tuples of contractions differ by a tuple of commuting elements.

Let $\At_n=(A_1,\ldots,A_n)$ and $\Bt_n=(B_1,\ldots,B_n)$ be $n$-tuples of pairwise commuting contractions and define $V_j=B_j-A_j$, $1\leq j\leq n$. Assume, in addition, that $\Vt_n=(V_1,\ldots,V_n)$ consists of commuting elements, which is equivalent to $\At_n$ and $\Bt_n$ being connected by a linear path of commuting contractions (see Lemma \ref{pathperturb}), and that for every $s\in [0,1]$, the tuple $\At_n+s\Vt_n$ has a commuting contractive normal dilation \eqref{dilfla}. For $\Vt_n$ consisting of elements of $\Ical$, we prove (see Theorem \ref{trthm}) existence of finite measures $\mu_j$, $1\leq j\leq n$, determined by $\At_n$ and $\Bt_n$ and satisfying
\begin{align*}
\tau_\Ical\big(f(\Bt_n)-f(\At_n)\big)
=\sum_{j=1}^n\int_{\T^n}\frac{\partial f }{\partial z_j}(z_1,\ldots,z_n)\,d\mu_j(z_1,\ldots,z_n).
\end{align*}
The measure obtained by projecting $\mu_j$ to the $j$-th component satisfies Theorem \ref{dsmt1} for the pair of contractions $A_j$ and $B_j$, $1\leq j\leq n$.

If we replace the requirement $V_j\in\Ical$ by $V_j\in\Ical^{1/2}$, $1\leq j\leq n$, where $\Ical^{1/2}$ satisfies \eqref{normCS}, then we prove (see Theorem \ref{tr2thm}) existence of finite measures $\nu_{ij}$, $1\leq i\leq j\leq n$, determined by $\At_n$ and $\Bt_n$ and satisfying
\begin{align*}
&\tau_\Ical\left(f(\Bt_n)-f(\At_n)-\frac{d}{ds}\bigg|_{s=0}f(\Xt_n(s))\right)
=\sum_{1\leq j\leq n}\int_{\T^n}\frac{\partial^2 f }{\partial z_j^2}(z_1,\ldots,z_n)\,d\nu_{jj}(z_1,\ldots,z_n)\\
&\quad\quad+2\sum_{1\leq i< j\leq n}
\int_{\T^n}\frac{\partial^2 f }{\partial z_i\partial z_j}(z_1,\ldots,z_n)\,d\nu_{ij}(z_1,\ldots,z_n).
\end{align*}
The projection of the measure $\nu_{jj}$ to the $j$-th component satisfies Theorem \ref{dsmt2} for the pair of contractions $A_j$ and $B_j$, $1\leq j\leq n$.

If $\At_n$ and $\Bt_n$ are tuples of commuting self-adjoint operators, then there exist finite measures $\mu_j$ and $\nu_{ij}$ such that the above trace formulas hold with integrals evaluated over $[-1,1]^n$ (see Theorems \ref{trthm} and \ref{tr2thm}).
In Corollaries \ref{arem1} and \ref{arem2}, we obtain more information about structure of the measures $\mu$ and $\nu$ provided by Theorems \ref{dsmt1} and \ref{dsmt2}, respectively, for a pair of normal operators $(H_0,H_0+V)$ satisfying $[\re(V),\im(V)]=0$ (which includes the case of purely real and purely imaginary perturbations).

The results follow from estimates for traces of the first and second order partial derivatives of multivariate operator functions that we establish in Theorems \ref{dthm} and \ref{d2thm}. We note that derivatives of single variable operator functions are fairly well explored (see, e.g., \cite{Peller,PSS,PSS-circle}), which is not the case with multivariate functions.
Existence of the first and higher order derivatives of multivariate functions along paths of tuples of commuting self-adjoint matrices in the operator norm was discussed in \cite{Bickel} and of the first order derivatives of functions along commuting tuples of self-adjoint operators in the Schatten $p$-norms, $1<p<\infty$, in \cite{KPSS}.

The operator derivatives $\frac{d}{ds}\big|_{s=0}f(\Xt_n(s))$ in \cite{KPSS} are evaluated in the norm of the Schatten ideal $S^p$, $1<p<\infty$, along paths of tuples of bounded commuting self-adjoints $\Xt_n(s)$ with tangent vectors in the closure of the narrow tangent space
\[\Gamma_p^0(\At_n)=\big\{(\i[A_1,Y]+Z_1,\ldots,\i[A_n,Y]+Z_n):\, Y\in S^p,\, Z\in\{\At_n\}''\cap S^p\big\},\]
where $\At_n=\Xt_n(0)$ and $\{\At_n\}''$ is a bicommutant of the family $\{A_1,\ldots,A_n\}$.
If $\Vt_n\in\Gamma_1^0(\At_n)$ and $f$ is analytic, then $\tau_{S^1}\left(\frac{d}{ds}\big|_{s=0}f(\Xt_n(s))\right)$ is trivial to handle and not enough for our goals. However, we consider only linear paths of operators and only analytic scalar functions $f$, what allows to handle more general $\Vt_n$ and non-self-adjoint operators (see Lemma \ref{dlemma} and Remark \ref{spremark}).

Despite all the commutativity assumptions, the second order derivatives of multivariate operator functions are more complex objects than the second order derivatives of single variable operator functions. If $f$ is a polynomial, $A$ a bounded operator, $V\in S^n$, and $\Tr$ a canonical trace, then $\Tr\big(\frac{d^n}{ds^n}f(A+sV)\big)=\Tr\big(\frac{d^{n-1}}{ds^{n-1}}f'(A+sV)\,V\big)$ (see, e.g., \cite[Lemma 2.2]{PSS-circle}). However, the analogous property fails to hold for multivariate functions already in the matrix case with all but one components of a perturbation equal to zero. For instance, if $A=\begin{pmatrix}1&0\\0&0\end{pmatrix}$, $V=\begin{pmatrix}0&1\\1&0\end{pmatrix}$, and $f(x_1,x_2)=x_1^3x_2$, then $\Tr\big(\frac{d^2}{ds^2}\big|_{s=0}f(A+sV,A)\big)=4$ while $\Tr\big(\frac{d}{ds}\big|_{s=0}(\frac{\partial f}{\partial x_1})(A+sV,A)\,V\big)=3$. In the multivariate case we also have mixed partial derivatives, which do not exist in the single variable case.

\section{Preliminaries and notations.}

\paragraph{\bf Multivariate operator functions.}
Denote by $\Ccal_n$ the set of all $n$-tuples of pairwise commuting (nonstrict) contractions acting on a separable Hilbert space $\mathcal{H}$. We have $H^\infty$ functional calculus for $\Ccal_n$, that is, if a function $f(z_1,\ldots,z_n)$ is holomorphic on a polydisc
\[\D^n_{1+\eps}=\{(z_1,\ldots,z_n)\in\C^n:\, |z_j|<1+\eps, 1\leq j\leq n\},\quad \eps>0,\]
and given by an absolutely convergent series
\begin{align}
\label{sr}
f(z_1,\ldots,z_n)=\sum_{k_1,\ldots,k_n\geq 0}c_{k_1,\ldots,k_n}z_1^{k_1}\ldots z_n^{k_n},\quad\text{for }
(z_1,\ldots,z_n)\in\D^n_{1+\eps},
\end{align}
then the operator function is defined by
\begin{align}
\label{opf}
f(A_1,\ldots,A_n)=\sum_{k_1,\ldots,k_n\geq 0}c_{k_1,\ldots,k_n}A_1^{k_1}\ldots A_n^{k_n},\quad\text{for }
(A_1,\ldots,A_n)\in\Ccal_n.
\end{align}
We will denote the space of functions holomorphic on $\D^n_{1+\eps}$ by $\Acal(\D^n_{1+\eps})$.
We will also consider functions representable by their Taylor series \eqref{sr} with $(z_1,\ldots,z_n)\in(-1-\eps,1+\eps)^n$. The respective space of functions is denoted by $\Acal((-1-\eps,1+\eps)^n)$.

Although the main results are obtained for functions of tuples of commuting contractions, some of the results hold for operators $f(A_1,\ldots,A_n)$ defined by the power series \eqref{opf} with $A_1,\ldots,A_n$ not necessarily commuting.

We will prove trace formulas for tuples of contractions $\Xt_n\in\Ccal_n$ satisfying the von Neumann inequality
\begin{align}
\label{vNineq}
\|f(\Xt_n)\|\leq \|f\|_{L^\infty(\T^n)},\quad f\in\AD,\quad\text{with }\eps>0.
\end{align}
The von Neumann inequality holds for any single contraction \cite{vN} and for any pair of commuting contractions \cite{Ando}. For any $n\in\N$, $n$-tuples of commuting normal contractions
$\Xt_n$ satisfy the von Neumann inequality. Another sufficient condition for an $n$-tuple of commuting contractions to satisfy \eqref{vNineq} is established in \cite{KV}.

Recall that for $\Xt_n$ a tuple of bounded commuting normal operators and $f$ a bounded Borel function on $\sigma(\Xt_n)$, the joint spectrum of the operators $X_1,\ldots,X_n$ (which is a subset of $\sigma(X_1)\times\ldots\times\sigma(X_n)$), the operator function $f(\Xt_n)$ can be represented as an integral with respect to the joint spectral measure $E$ of the tuple
\[f(\Xt_n)=\idotsint\limits_{\sigma(\Xt_n)}f(z_1,\ldots,z_n)\,dE(z_1,\ldots,z_n).\]
The measure $E$ is the product of the spectral measures $E_j$ of the operators $X_j$, $1\leq j\leq n$
\cite[Theorem 6.5.1 and Subsection 6.6.2]{BSbook}. In particular,
\[E(\delta_1\times\ldots\times\delta_n)=E_1(\delta_1)\ldots E_n(\delta_n),\]
where $\delta_1,\ldots,\delta_n$ are Borel subsets of $\C$.
The measure $E$ is supported in $\sigma(\Xt_n)$.

We will prove our results for tuples of commuting contractions that can be dilated to tuples of commuting normal contractions. That is, we will consider those $\Xt_n\in\Ccal_n$ for which there exists a Hilbert space $\mathcal{K}\supset\mathcal{H}$ and a tuple of commuting normal contractions $(U_1,\ldots,U_n)$ on $\mathcal{K}$ such that
\begin{align}
\label{dilfla}
X_{j_1}\ldots X_{j_l}= P_\mathcal{H} U_{j_1}\ldots U_{j_l}|_{\mathcal{H}},\quad l\in\N,\; 1\leq j_1,\ldots,j_l \leq n,
\end{align}
where $P_{\mathcal{H}}$ is the orthogonal projection from $\mathcal{K}$ onto $\mathcal{H}$. If $\Xt_n\in\Ccal_n$ satisfies \eqref{dilfla}, then it also satisfies \eqref{vNineq}. It is well known that a unitary dilation exists for a single contraction \cite{Nagy} and a commuting unitary dilation exists for a pair of commuting contractions \cite{Ando}. A necessary and sufficient condition for existence of a multivariate commuting unitary dilation \eqref{dilfla}, which is formulated in terms of the von Neumann inequalities for matrix-valued functions, can be found in \cite[Corollary 4.9]{Pisier}.

Throughout the paper, we assume the following notations.
\begin{notation}
\label{not}
All operators are assumed to be elements of $\BH$, unless stated otherwise.
\begin{enumerate}[(i)]
\item Denote the tuple of bounded operators $A_1,\ldots,A_n$ by
\begin{align*}
\At_n=(A_1,\ldots,A_n).
\end{align*}

\item 
Let $\Xt_n$ be an $n$-tuple of bounded operators and let $k_1,\ldots,k_n\in\N\cup\{0\}$. Denote
\begin{align*}
&T_{k_1,\ldots,k_n}(\Xt_n)=X_1^{k_1}\ldots X_n^{k_n},\\
&T_{k_m,\ldots,k_n}(_{m}\Xt_n)=X_m^{k_m}\ldots X_n^{k_n},\quad 1<m\leq n.
\end{align*}


\item For $\At_n,\Bt_n\in\Ccal_n$, denote
\begin{align*}
&V_j=B_j-A_j,\quad X_j(s)=A_j+sV_j,\quad 1\leq j\leq n,\;s\in[0,1],\\
&\Xt_n(s)=(X_1(s),\ldots,X_n(s)).
\end{align*}
\item Let $\Ncal_n$ denote the set of pairs $\{\At_n,\Bt_n\}$ in $\Ccal_n$ such that
the linear path joining $\At_n$ and $\Bt_n$ consists of tuples of commuting contractions admitting a commuting contractive normal dilation; that is, $\{\At_n,\Bt_n\}\in\Ncal_n$ if and only if $(A_1+tV_1,\ldots,A_n+tV_n)\in\Ccal_n$ satisfies \eqref{dilfla} (and, hence, satisfies \eqref{vNineq}) for every $t\in [0,1]$.
\end{enumerate}
\end{notation}

Note that if $\At_n,\Bt_n\in\Ccal_n$, then $X_j(t)=A_j(1-t)+tB_j$ is a contraction for every $t\in [0,1]$, $1\leq j\leq n$.
\bigskip

\paragraph{\bf Normed ideals.}
Our perturbations will be elements of symmetrically normed ideals.

\begin{defi}
\label{ideal}
An ideal $\Ical$ of $\BH$ is called a normed ideal if it is equipped with an ideal norm, namely, a norm $\|\cdot\|_\Ical$ satisfying
\begin{enumerate}[(i)]
\item $A\in\BH,B\in\Ical$, $0\leq A\leq B$ implies $\|A\|_\Ical\leq \|B\|_\Ical$,

\item there is a constant $K>0$ such that $\|B\|\leq K\|B\|_\Ical$ for all $B\in\Ical$,

\item for all $A,C\in\BH$ and $B\in\Ical$, we have $\|ABC\|_{\Ical}\leq\|A\|\|B\|_\Ical\|C\|$.
\end{enumerate}
\end{defi}

A trace $\tau_\Ical$ is $\|\cdot\|_\Ical$-bounded if there is a constant $M>0$ such that $|\tau_\Ical(A)|\leq M\|A\|_\Ical$, for every $A\in\Ical$. The infimum of such constants $M$ equals $\|\tau_\Ical\|_{\Ical^*}$.

We will also consider perturbations in the ideal $\Ical^{1/2}=\{A\in\BH:\, |A|^2\in\Ical\}$, which contains the ideal $\Ical$. For a positive trace $\tau_\Ical$ on $\Ical$ we have the Cauchy-Schwarz inequality
\[|\tau_\Ical(AB)|\leq\big(\tau_\Ical(|A|^2)\big)^{1/2}\big(\tau_\Ical(|B|^2)\big)^{1/2},\quad
A,B\in\Ical^{1/2}.\]
A sufficient condition for $\Ical^{1/2}$ to be a normed ideal with norm $\|A\|_{\Ical^{1/2}}=\| |A|^2 \|_{\Ical}^{1/2}$ and for the inequality
\begin{align}
\label{normCS}
\|AB\|_\Ical\leq\|A\|_{\Ical^{1/2}}\|B\|_{\Ical^{1/2}},\quad A,B\in\Ical^{1/2},
\end{align}
to hold is established in \cite[Proposition 2.5]{dsD}.

Working with contractions, we will use dilations to unitary operators. Normed ideals and bounded traces can also be dilated.

\begin{prop}\cite[Proposition 2.3]{dsD}
\label{dilation}
Let $\Ical$ be a normed ideal of $\BH$ with an ideal norm $\|\cdot\|_\Ical$ and a positive $\|\cdot\|_\Ical$-bounded trace $\tau_\Ical$. If $\mathcal{K}$ is a separable Hilbert space containing $\mathcal{H}$, then there are
\begin{enumerate}[(i)]
\item an ideal $\tilde\Ical$ of $\mathcal{B}(\mathcal{K})$ such that $\tilde\Ical\cap \mathcal{B}(\mathcal{H})=\Ical$,

\item an ideal norm $\|\cdot\|_{\tilde\Ical}$ on $\tilde\Ical$ whose restriction to $\Ical$ equals $\|\cdot\|_\Ical$,

\item a positive $\|\cdot\|_{\tilde\Ical}$-bounded trace $\tau_{\tilde\Ical}$ on $\tilde\Ical$ whose restriction to $\Ical$ equals $\tau_\Ical$.
\end{enumerate}
\end{prop}

\section{First order derivatives.}

\begin{lemma}
Let $H_0$ and $V$ be bounded operators and $p$ a natural number. Then,
\begin{align}
\label{mondif}
(H_0+V)^p-H_0^p=\sum_{\substack{p_{0},\,p_{1}\geq 0\\p_{0}+p_{1}=p-1}}(H_0+V)^{p_0}VH_0^{p_1},
\end{align}
\begin{align}
\label{monder}
\frac{d}{ds}\bigg|_{s=t}(H_0+sV)^p=\sum_{\substack{p_{0},\,p_{1}\geq 0\\p_{0}+p_{1}=p-1}}(H_0+tV)^{p_0}V(H_0+tV)^{p_1},
\end{align}
\begin{align}
\label{monder2}
\frac{d^2}{ds^2}\bigg|_{s=t}(H_0+sV)^p=
2\!\!\sum_{\substack{p_{0},\,p_{1},\,p_{2}\geq 0\\p_{0}+p_{1}+p_{2}=p-2}}\!\!
(H_0+tV)^{p_0}V(H_0+tV)^{p_1}V(H_0+tV)^{p_2},\quad p\geq 2,
\end{align}
where the derivatives exist in the operator norm.
If, in addition, $V\in\Ical$, then the derivative in \eqref{monder} exists in the ideal norm $\|\cdot\|_\Ical$.
\end{lemma}

\begin{proof}
The representation \eqref{mondif} is routine. We will prove existence of the derivative of an operator polynomial in the ideal norm and the representation for the derivative \eqref{monder}. Existence of the derivative in the operator norm in case $V\notin\Ical$ can be established completely analogously.

Applying the representation \eqref{mondif} twice gives
\begin{align*}
&h(\vareps)=\frac{(H_0+(t+\vareps)V)^p-(H_0+tV)^p}{\vareps}-\sum_{\substack{p_{0},\,p_{1}\geq 0\\p_{0}+p_{1}=p-1}}(H_0+tV)^{p_0}V(H_0+tV)^{p_1}\\
&=\sum_{\substack{p_{0},\,p_{1}\geq 0\\p_{0}+p_{1}=p-1}}\,
\sum_{\substack{q_{0},\,q_{1}\geq 0\\q_{0}+q_{1}=p_0-1}}(H_0+(t+\vareps)V)^{q_0}V(H_0+tV)^{q_1}
V(H_0+tV)^{p_1}.
\end{align*}
The ideal norm of the latter expression is estimated by the H\"{o}lder inequality ensuring
\begin{align*}
\|h(\vareps)\|_\Ical\leq p(p-1)\vareps\|V\|\|V\|_\Ical,
\end{align*} which proves \eqref{monder} for the derivative evaluated in the ideal norm.
The completely analogous reasoning gives \eqref{monder2}.
\end{proof}

The derivative of $f$ along the direction ${\bf V}_n$ can be computed via partial derivatives along directions $V_1,\ldots,V_n$.

\begin{lemma}
\label{dlemma}
Assume Notations \ref{not} and let $f\in \Acal(\D^n_{1+\eps})$.
\begin{enumerate}[(i)]
\item
The derivative $\ds t\mapsto\frac{d}{ds}\bigg|_{s=t}f(\Xt_n(s))$ exists and is continuous on $[0,1]$ in the operator norm. Moreover, for $f$ given by \eqref{sr},
\begin{align}
\label{dfla}
\frac{d}{ds}\bigg|_{s=t}f(\Xt_n(s))=\sum_{j=1}^n D^j_f(t),
\end{align}
where
\begin{align*}
D^j_f(t)=\!\!\sum_{k_1,\ldots,k_n\geq 0}\!\!c_{k_1,\ldots,k_n}T_{k_1,\ldots,k_{j-1}}(\Xt_{j-1}(t))\,
\frac{d}{ds}\bigg|_{s=t}X_j(s)^{k_j}\,T_{k_{j+1},\ldots,k_n}(_{j+1}\Xt_n(t)).
\end{align*}
\item
Let $\Ical$ be a normed ideal.
If $B_j-A_j\in\Ical$, then $\ds t\mapsto\frac{d}{ds}\bigg|_{s=t}f(\Xt_n(s))$ exists and is continuous on $[0,1]$ in the ideal norm $\|\cdot\|_\Ical$.
\end{enumerate}
\end{lemma}

\begin{proof}
We will only establish existence and continuity of $\ds t\mapsto\frac{d}{ds}\bigg|_{s=t}f(\Xt_n(s))$ in the ideal norm; the results for the operator norm can be established completely analogously.
Denote
\begin{align}
\label{gdef}
g_{j,p_0^j,p_1^j}(t)=
T_{k_1,\ldots,k_{j-1},p_0^j}(\Xt_j(t))\,V_j\,T_{p_1^j,k_{j+1},\ldots,k_n}(_{j}\Xt_n(t)),
\end{align}
so
\begin{align*}
&D_f^j(t)=\sum_{k_1,\ldots,k_n\geq 0}c_{k_1,\ldots,k_n}\!\!
\sum_{\substack{p_{0}^j,\,p_{1}^j\geq 0\\p_{0}^j+p_{1}^j=k_j-1}}\!\!g_{j,p_0^j,p_1^j}(t).
\end{align*}
We obtain
\begin{align*}
&\sup_{t\in [0,1]}\|D_f^j(t)\|_\Ical
\leq \max_{1\leq j\leq n}\|V_j\|_\Ical\sum_{k_1,\ldots,k_n\geq 0}(k_1+\ldots+k_n)\,|c_{k_1,\ldots,k_n}|,
\end{align*}
where the series converges because every partial derivative of $f$ exists in $\D^n_{1+\eps}$ and is given by the absolutely convergent series obtained by termwise differentiation of the series \eqref{sr}. In case $k_j=0$, the sum $\sum_{\substack{p_{0}^j+p_{1}^j=k_j-1}}$ is empty.

For simplicity of exposition, we will establish \eqref{dfla} only for $t=0$, as the same method works for any value of $t$.
Applying \eqref{mondif} to each term of the series representing $f(\Xt_n(t))$ gives
\begin{align}
\label{multest}
\frac{f(\Xt_n(\vareps))-f(\At_n)}{\vareps}
=\sum_{j=1}^n\sum_{k_1,\ldots,k_n\geq 0}c_{k_1,\ldots,k_n}\!\!
\sum_{\substack{p_{0}^j,\,p_{1}^j\geq 0\\p_{0}^j+p_{1}^j=k_j-1}}\!\!h_{j,p_0^j,p_1^j}(\vareps),
\end{align}
\begin{align*}
h_{j,p_0^j,p_1^j}(\vareps)=T_{k_1,\ldots,k_{j-1}}(\At_{j-1})\, X_j(\vareps)^{p_{0}^j}V_j
A_j^{p_{1}^j}\,T_{k_{j+1},\ldots,k_n}(_{j+1}\Xt_n(\vareps)).
\end{align*}
Applying \eqref{mondif} gives
\begin{align}
\label{star}
&h_{j,p_0^j,p_1^j}(\vareps)-g_{j,p_0^j,p_1^j}(0)\\
\nonumber
&=\vareps T_{k_1,\ldots,k_{j-1}}(\At_{j-1})\bigg(\!\!\sum_{\substack{q_{0}^j,\,q_{1}^j\geq 0\\q_{0}^j+q_{1}^j=p_0^j-1}}\!\!X_j(\vareps)^{q_0^j}V_jA_j^{q_1^j}V_j
A_j^{p_1^j}\,T_{k_{j+1},\ldots,k_n}(_{j+1}\Xt_n(\vareps))\\
\nonumber
&\quad+A_j^{p_0^j}V_j\sum_{i=j+1}^n
T_{p_1^j,k_{j+1},\ldots,k_{i-1}}(_{j}\At_{i-1})\!\!\!
\sum_{\substack{q_{0}^i,\,q_{1}^i\geq 0\\q_{0}^i+q_{1}^i=k_i-1}}\!\!
X_i(\vareps)^{q_0^i}V_iA_i^{q_1^i}\,T_{k_{i+1},\ldots,k_n}(_{i+1}\Xt_n(\vareps))\bigg).
\end{align}
Hence,
\begin{align}
\label{h-g}
\big\|h_{j,p_0^j,p_1^j}(\vareps)-g_{j,p_0^j,p_1^j}(0)\big\|_\Ical\leq n\vareps K\max_{1\leq j\leq n}\|V_j\|_\Ical^2\,(k_1+\ldots+k_n).
\end{align}
Combining \eqref{multest} and \eqref{h-g} gives
\begin{align*}
&\bigg\|\frac{f(\Xt_n(\vareps))-f(\At_n)}{\vareps}
-\sum_{j=1}^n D_f^j(0)\bigg\|_\Ical\\
&\quad \leq n^2\vareps K\max_{1\leq j\leq n}\|V_j\|_\Ical^2\sum_{k_1,\ldots,k_n\geq 0}(k_1+\ldots+k_n)^2\,|c_{k_1,\ldots,k_n}|.
\end{align*}
Thus, the G\^{a}teaux derivative of $f$ along the direction $(V_1,\ldots,V_n)$ at the point $(A_1,\ldots,A_n)$ exists in the norm $\|\cdot\|_\Ical$ and is given by \eqref{dfla}.

Continuity of $\ds t\mapsto\frac{d}{ds}\bigg|_{s=t}f(\Xt_n(s))$ can be established by applying \eqref{mondif} as it was done in the proof of its existence.
\end{proof}

\begin{hyp}
\label{hyp}
Let $\Ical$ be a normed ideal endowed with a positive and $\|\cdot\|_\Ical$-bounded trace $\tau_\Ical$.
Assume that $\At_n,\Bt_n\in\Ccal_n$ satisfy $B_j-A_j\in\Ical$, $1\leq j\leq n$.
\end{hyp}

Immediately from Lemma \ref{dlemma}, we have the following analog of the fundamental theorem of calculus.

\begin{cor}
\label{ftc}
Assume Notations \ref{not} and Hypotheses \ref{hyp} and let $f\in \Acal(\D^n_{1+\eps})$. Then,
\[\tau_\Ical\big(f(\Bt_n)-f(\At_n)\big)=
\int_0^1\tau_\Ical\left(\frac{d}{ds}\bigg|_{s=t}f(\Xt_n(s))\right)\,dt.\]
\end{cor}

By adjusting the argument in the proof of \cite[Lemma 5.8]{dsD}, we obtain the following simple lemma.

\begin{lemma}
\label{m2v}
Let $n_1,n_2, N_1, N_2 \in\N$.
Let $(\delta_{i_1})_{1\leq i_1\leq N_1}$ and $(\delta_{i_2})_{1\leq i_2\leq N_2}$ be partitions of $\C^{n_1}$ and $\C^{n_2}$, respectively, and let $E_1$ and $E_2$ be spectral measures on $\C^{n_1}$ and $\C^{n_2}$, respectively.
\begin{enumerate}[(i)]
\item\label{m2vi} If $V\in\Ical$, then
\[\sum_{i_1=1}^{N_1}\big|\tau_\Ical\big(E_1(\delta_{i_1})V\big)\big|
\leq\tau_\Ical\big(|\re(V)|+|\im(V)|\big).\]
\item\label{m2vii} If $V_1,V_2\in\Ical^{1/2}$, then
\[\sum_{i_1=1}^{N_1}\sum_{i_2=1}^{N_2}\big|\tau_\Ical
\big(E_1(\delta_{i_1})V_1E_2(\delta_{i_2})V_2E_1(\delta_{i_1})\big)\big|\leq \big(\tau_\Ical(|V_1|^2)\big)^{1/2}\big(\tau_\Ical(|V_2|^2)\big)^{1/2}.\]
\end{enumerate}
\end{lemma}

Under an additional commutativity assumption, the trace of a derivative of a multivariate operator function can be written via partial derivatives of $f$.

\begin{thm}
\label{dthm}
Assume Notations \ref{not} and Hypotheses \ref{hyp}.
Suppose that there exists $t\in [0,1]$ such that $\Xt_n(t)\in\Ccal_n$ satisfies \eqref{dilfla}.
Then, for every $f\in \Acal(\D_{1+\eps}^n)$, $\eps>0$,
\begin{align}
\label{chain}
\tau_\Ical\left(\frac{d}{ds}\bigg|_{s=t}f(\Xt_n(s))\right)
=\sum_{j=1}^n\tau_\Ical\left(V_j\,\frac{\partial{f}}{\partial{z_j}}(\Xt_n(t))\right)
\end{align}
and
\begin{align}
\label{pdest}
\left|\tau_\Ical\left(V_j\,\frac{\partial{f}}{\partial{z_j}}(\Xt_n(t))\right)\right|\leq \min\big\{\|\tau_\Ical\|_{\Ical^*}\cdot\|V_j\|_\Ical,\,\tau_\Ical\big(|\re(V_j)|+|\im(V_j)|\big)\big\} \left\|\frac{\partial{f}}{\partial{z_j}}\right\|_{L^\infty(\Omega^n)},
\end{align}
where $\Omega=\T$. If, in addition, $\Xt_n(t)$ is a tuple of self-adjoint contractions, then \eqref{pdest} holds with $\Omega=[-1,1]$ for $f\in \Acal((-1-\eps,1+\eps)^n)$, $\eps>0$.
\end{thm}

\begin{proof}
We provide the proof only in case $\Omega=\T$; the proof in case $\Omega=[-1,1]$ is a verbatim repetition.

Applying Lemma \ref{dlemma} gives
\begin{align*}
\tau_\Ical\left(\frac{d}{ds}\bigg|_{s=t}f(\Xt_n(s))\right)=\sum_{j=1}^n\sum_{k_1,\ldots,k_n\geq 0}c_{k_1,\ldots,k_n}\!\!\sum_{\substack{p_{0}^j,\,p_{1}^j\geq 0\\p_{0}^j+p_{1}^j=k_j-1}}\!\!\tau_\Ical\left(g_{j,p_0^j,p_1^j}(t)\right),
\end{align*}
where $g_{j,p_0^j,p_1^j}$ is defined by \eqref{gdef}.
Cyclicity of the trace and pairwise commutativity of $X_1(t),\ldots,X_n(t)$ give
\begin{align*}
\tau_\Ical\left(g_{j,p_0^j,p_1^j}(t)\right)=\tau_\Ical\left(V_j\,
T_{k_1,\ldots,k_{j-1},k_j-1,k_{j+1},\ldots,k_n}(\Xt_n(t))\right).
\end{align*}
Thus,
\begin{multline*}
\tau_\Ical\left(\frac{d}{ds}\bigg|_{s=t}f(\Xt_n(s))\right)
=\sum_{j=1}^n\sum_{k_1,\ldots,k_n\geq 0}k_j c_{k_1,\ldots,k_n}
\tau_\Ical\left(V_j\,T_{k_1,\ldots,k_{j-1},k_j-1,k_{j+1},\ldots,k_n}(\Xt_n(t))\right),
\end{multline*}
which equals \eqref{chain}. By the H\"{o}lder and von Neumann inequalities,
\[\left|\tau_\Ical\left(V_j\,\frac{\partial{f}}{\partial{z_j}}(\Xt_n(t))\right)\right|\leq \|\tau_\Ical\|_{\Ical^*}\cdot \|V_j\|_\Ical \left\|\frac{\partial{f}}{\partial{z_j}}\right\|_{L^\infty(\T^n)}.\]
Now we will complete the proof of \eqref{pdest}.

We suppose first that $\Xt_n(t)$ is a tuple of normal contractions and $E_t$ is its joint spectral measure. By the spectral theorem,
\begin{align*}
\frac{\partial f}{\partial z_j}(\Xt_n(t))
&=\int_{\C^n}\frac{\partial f}{\partial z_j}(z_1,\ldots,z_n)\,dE_t(z_1,\ldots,z_n).
\end{align*}
Then, for every $1\leq l\leq n$, there is a sequence of Borel partitions $(\delta_{m,l,\beta_l})_{1\leq \beta_l\leq m}$ of $\C$ and a sequence of tuples of complex numbers $(z_{m,l,\beta_l})_{1\leq \beta_l\leq m}$ such that
\begin{align*}
&\tau_\Ical\left(V_j\frac{\partial f}{\partial z_j}(\Xt_n(t))\right)\\
&\quad=\lim_{m\rightarrow\infty}\sum_{1\leq \beta_1,\ldots,\beta_n\leq m}
\frac{\partial f}{\partial z_j}(z_{m,1,\beta_1},\ldots,z_{m,n,\beta_n})
\tau_\Ical\big(E_t(\delta_{m,1,\beta_1}\times\ldots\times\delta_{m,n,\beta_n})V_j\big).
\end{align*}
Applying Lemma \ref{m2v}\eqref{m2vi} concludes the proof of \eqref{pdest} in case of normal contractions.

Let now $\Xt_n(t)\in\Ccal_n$. Since there exists a tuple of normal contractions $\Ut_n$ on a Hilbert space $\mathcal{K}\supset\mathcal{H}$ such that \eqref{dilfla} holds, we have
\begin{align*}
\left|\tau_\Ical\left(V_j\frac{\partial f}{\partial z_j}(\Xt_n(t))\right)\right|
=\left|\tau_{\tilde\Ical}\left(P_{\mathcal{H}}\frac{\partial f}{\partial z_j}(\Ut_n)V_j\right)\right|
\leq \left|\tau_{\tilde\Ical}\left(\frac{\partial f}{\partial z_j}(\Ut_n)V_j\right)\right|,
\end{align*}
for $\tau_{\tilde\Ical}$ given by Proposition \ref{dilation}. Hence,
\eqref{pdest} follows by the just established estimate for a tuple of normal contractions.
\end{proof}

\begin{remark}
\label{spremark}
Combining the result of Lemma \ref{dlemma} on existence of derivatives of multivariate functions with the results of \cite[Corollary]{KPSS} and \cite[Theorem 5.2]{CMPS} on operator Lipschitzness of scalar functions, we conclude the following. Assume that $\At_n$ and $\Vt_n$ are paths of self-adjoint contractions. If $V_j\in S^p$, $1<p<\infty$, $1\leq j\leq n$, and if $\Xt_n(s)=\At_n+s\Vt_n\in\Ccal_n$, $s\in[0,1]$, then
\begin{align*}
\left\|\frac{d}{ds}\bigg|_{s=0}f(\Xt_n(s))\right\|_p\leq \frac{C_n\, p^2}{p-1}\cdot\sup_{\vec x, \vec y \in \cup_{t\in [0,1]}\sigma(\Xt_n(t))}\!\frac{|f(\vec x)-f(\vec y)|}{\|\vec x-\vec y\|_1}\cdot\sum_{j=1}^n\|V_j\|_p.
\end{align*}
\end{remark}




\begin{lemma}
\label{pathperturb}
The following statements are equivalent.
\begin{enumerate}[(i)]
\item\label{(i)} $(A_1+t(B_1-A_1),\ldots,A_n+t(B_n-A_n))\in\Ccal_n$, for every $t\in[0,1]$, $1\leq j\leq n$.
\item\label{(ii)} $\At_n,\Bt_n\in\Ccal_n$ and $[B_i-A_i,B_j-A_j]=0$, for all $1\leq i, j\leq n$.
\end{enumerate}
\end{lemma}

\begin{proof}
Let $h_{ij}(t)=[A_i+t(B_i-A_i),A_j+t(B_j-A_j)]$. 
It is straightforward to see that
\begin{align}
\label{hij1}
h_{ij}(t)=(1-t)^2[A_i,A_j]+t^2[B_i,B_j]+t(1-t)[A_i,B_j]+t(1-t)[B_i,A_j].
\end{align}
Note that \eqref{(i)} is equivalent to $h_{ij}(t)=0$, for every $t\in[0,1]$, $1\leq j\leq n$. Thus, \eqref{(ii)} immediately implies \eqref{(i)}. Since \eqref{(ii)} is equivalent to $[A_i,A_j]=[B_i,B_j]=[A_i,B_j]+[B_i,A_j]=0$, $1\leq i\neq j\leq n$,
by differentiating \eqref{hij1} with respect to $t$ twice and considering $h_{ij}(0)$ and $h_{ij}(1)$, we see that \eqref{(i)} implies \eqref{(ii)}.
\end{proof}

\begin{thm}
\label{trthm}
Assume Notations \ref{not}, Hypotheses \ref{hyp}, and $[V_i,V_j]=0$, for $1\leq i\neq j\leq n$. If $\{\At_n,\Bt_n\}\in\Ncal_n$,
then there exist finite measures $\mu_1,\ldots,\mu_n$ on $\Omega^n$ such that for every $1\leq j\leq n$,
\begin{align}
\label{min}
\|\mu_j\|\leq \min\big\{\|\tau_\Ical\|_{\Ical^*}\cdot\|V_j\|_\Ical,\,\tau_\Ical\big(|\re(V_j)|+|\im(V_j)|\big)\big\}
\end{align}
\begin{align}
\label{mutau}
\mu_j(\Omega^n)=\tau_\Ical(V_j),
\end{align}
and
\begin{align}
\label{tf}
\tau_\Ical\big(f(\Bt_n)-f(\At_n)\big)
=\sum_{j=1}^n\int_{\Omega^n}\frac{\partial f }{\partial z_j}(z_1,\ldots,z_n)\,d\mu_j(z_1,\ldots,z_n),
\end{align}
for every $f\in\AD$, where $\Omega=\T$.

For every $1\leq j\leq n$ and $g\in\Acal(\D_{1+\eps})$,
\begin{align}
\label{connection1}
\int_{\Omega^n}g(z_j)\,d\mu_j(z_1,\ldots,z_n)=\int_{\Omega}g(z)\,d\mu_{A_j,B_j}(z),
\end{align}
where $\mu_{A_j,B_j}$ is the measure given by Theorem \ref{dsmt1} for the pair $(A_j,B_j)$.

If, in addition, $\At_n$ and $\Bt_n$ are tuples of self-adjoint contractions, then there exist  real-valued measures $\mu_j$, $1\leq j\leq n$, such that \eqref{min}--\eqref{connection1} hold with $\Omega=[-1,1]$ for $f(z_1,\ldots,z_n)\in \Acal((-1-\eps,1+\eps)^n)$ and $g\in\Acal((-1-\eps,1+\eps))$, $\eps>0$.
\end{thm}

\begin{proof}
Combining the results of Corollary \ref{ftc} and Theorem \ref{dthm}, we obtain
\begin{align}
\label{tinside}
\tau_\Ical\big(f(\Bt_n)-f(\At_n)\big)
=\sum_{j=1}^n\int_0^1\tau_\Ical\left(V_j\,
\frac{\partial{f}}{\partial{z_j}}(\Xt_n(t))\right)dt.
\end{align}
By the Riesz-Markov representation theorem for a bounded linear functional on the space of continuous functions on a compact set and by the Hahn-Banach theorem, we deduce from \eqref{tinside} and \eqref{pdest} existence of measures $\mu_j$ satisfying \eqref{min} and \eqref{tf}. Applying \eqref{tf} to $f(z_1,\ldots,z_n)=z_j$, $1\leq j \leq n$, gives \eqref{mutau} and applying \eqref{tf} to $f(z_1,\ldots,z_n)=f(z_j)$ such that $\frac{\partial f}{\partial z_j}=g$ gives \eqref{connection1}, $1\leq j \leq n$.
\end{proof}

\begin{cor}
\label{arem1}
Let $(A_1,A_2),(B_1,B_2)$ be tuples of commuting self-adjoint operators such that $\|A_1+\i A_2\|\leq 1$, $\|B_1+\i B_2\|\leq 1$, $V_1=B_1-A_1, V_2=B_2-A_2\in\Ical$, and $[V_1,V_2]=0$. Then, there exist real-valued measures $\mu_1,\mu_2$ satisfying \eqref{min}--\eqref{connection1} with $\Omega=[-1,1]$ and supported in the closed unit disc $\overline\D$.
\begin{enumerate}[(i)]
\item \label{arem1ii} For $f\in\Acal(\D_{1+\eps})$, $\eps>0$,
\[\tau_\Ical\big(f(B_1+\i B_2)-f(A_1+\i A_2)\big)=\int_{\overline\D}f'(z)\,d(\mu_1+\i\mu_2)(z).\]

\item \label{arem1iii}

If $A_i=B_i$ for $i\in\{1,2\}$, then $\mu_i=0$ and, for $g\in\Acal((-1-\eps,1+\eps))$, $\eps>0$,
\begin{align}
\label{con}
\int_{\overline\D}g(x_i)\,d\mu_j(x_1,x_2)=\tau\big(g(A_i)V_j\big).
\end{align}
\end{enumerate}
\end{cor}

\begin{proof}
Examination of the proof of Theorem \ref{trthm} shows that $\mu_1$ and $\mu_2$ satisfying \eqref{min}--\eqref{connection1} can be chosen to be supported in a compact subset of $\overline{\D}$ containing $\cup_{t\in [0,1]}\sigma(\Xt_2(t))$. The property \eqref{arem1ii} is straightforward.
The representation \eqref{con} follows from \eqref{tf} applied to $f(x_1,x_2)=g(x_i)x_j$, $1\leq i\neq j\leq 2$.
\end{proof}

Below, we relax commutativity assumptions of Theorem \ref{trthm} in case $\tau_\Ical$ is a singular trace.

\begin{prop}
\label{Dixthm}
Suppose $\tau_\Ical(\Ical^2)=\{0\}$ and assume Notations \ref{not} and Hypotheses \ref{hyp}. If $\At_n$ satisfies \eqref{dilfla}, then there exist finite measures $\mu_1,\ldots,\mu_n$ on $\T^n$ such that \eqref{min}--\eqref{connection1} hold.
\end{prop}

\begin{proof}
By Lemma \ref{dlemma} and the representation \eqref{multest} with $\vareps=1$, we have
\begin{align}
\label{beforetr}
&f(\Bt_n)-f(\At_n)-\frac{d}{ds}\bigg|_{s=0}f(\Xt_n(s))\\
\nonumber
&=\sum_{j=1}^n\sum_{k_1,\ldots,k_n\geq 0}c_{k_1,\ldots,k_n}\!\!
\sum_{\substack{p_{0}^j,\,p_{1}^j\geq 0\\p_{0}^j+p_{1}^j=k_j-1}}\!\!
\bigg(T_{k_1,\ldots,k_{j-1}}(\At_{j-1})\,B_j^{p_{0}^j}V_j
A_j^{p_{1}^j}\,T_{k_{j+1},\ldots,k_n}(_{j+1}\Bt_n)\\
\nonumber
&\quad\quad\quad
-T_{k_1,\ldots,k_{j-1},p_0^j}(\At_j)\,V_j\,T_{p_1^j,k_{j+1},\ldots,k_n}(_{j}\At_n)\bigg).
\end{align}
Further application of \eqref{mondif} as in \eqref{star} shows that the expression in \eqref{beforetr} is an element of $\Ical^2$. Therefore,
\begin{align*}
\tau_\Ical\big(f(\Bt_n)-f(\At_n)\big)
=\tau_\Ical\left(\frac{d}{ds}\bigg|_{s=0}f(\Xt_n(s))\right).
\end{align*}
By Theorem \ref{dthm}, the latter equals
\[\tau_\Ical\big(f(\Bt_n)-f(\At_n)\big)=
\sum_{j=1}^n\tau_\Ical\left(V_j\,\frac{\partial{f}}{\partial{z_j}}(\At_n)\right).\]
Repeating the reasoning in the proof of Theorem \ref{trthm} completes the proof of existence of measures satisfying \eqref{min}--\eqref{tf}.
\end{proof}



\section{Second order derivatives}

A natural object appearing in evaluation of directional derivatives of single variable operator functions is a divided difference. In representations of derivatives of multivariate operator functions, we will additionally need a certain modification of a multivariate difference operator.

Let $g$ be a function on $\C$. The $0$-th order divided difference of $g$ is defined by $g[\la_0]=g(\la_0)$; the divided difference of $k$-th order is defined recursively by
\[g[\la_0,\ldots,\la_r]=\begin{cases}\frac{g[\la_0,\ldots,\la_r]-g[\la_0,\ldots,\la_{r-1}]}
{\la_r-\la_{r-1}} &\text{ if } \la_{r-1}\neq\la_r\\
\frac{d}{dt}\big|_{t=\la_r}g[\la_0\ldots,\la_{r-2},t] &\text{ if }  \la_{r-1}=\la_r,\end{cases}\]
where $\la_0,\ldots,\la_r\in\C$. Note that $g[\la_0,\ldots,\la_r]$ is a symmetric function of the sequence $\{\la_0,\ldots,\la_r\}$. Let $\phi$ be a function on $\C^n$.
We define the $r$-th order divided difference of $\phi$ in the $j$-th coordinate by
\[\phi(z_1,\ldots,z_{j-1},[\la_0,\ldots,\la_r],z_{j+1},\ldots,z_n)=g_j[\la_0,\ldots,\la_r],\]
where
\begin{align}
\label{display}
g_j(\la)=\phi(z_1,\ldots,z_{j-1},\la,z_{j+1},\ldots,z_n)
\end{align}
and $z_1,\ldots,z_{j-1},z_{j+1},\ldots,z_n$ are fixed points in $\C$.

The following well known property of the divided difference can be established by induction on the order $r$.

\begin{lemma}
\label{ddder}
If $\phi\in\Acal(\C^n)$, $z_1,\ldots,z_{j-1},z_{j+1},\ldots,z_n\in\C$, and $g_j$ is given by \eqref{display}, then
\begin{align*}
&\phi(z_1,\ldots,z_{j-1},[\la_0,\ldots,\la_r],z_{j+1},\ldots,z_n)\\
&=\int_0^1 dt_1\int_0^{t_1}dt_2\ldots\int_0^{t_{r-1}}g_j^{(r)}
(\la_r+(\la_{r-1}-\la_r)t_1+\ldots+(\la_0-\la_1)t_r)\,dt_r,
\end{align*}
and
\[\sup_{z_1,\ldots,z_n,\la_0,\ldots,\la_r\in\C}\big|
\phi(z_1,\ldots,z_{j-1},[\la_0,\ldots,\la_r],z_{j+1},\ldots,z_n)\big|
\leq\frac{1}{r!}\left\|\frac{\partial^r \phi}{\partial z_j^r}\right\|_\infty.\]
\end{lemma}

Let $\{\vec e_1,\ldots,\vec e_n\}$ denote the canonical basis in $\R^n$. Let $0\neq h_i\in\C$, $1\leq i\leq n$, and $\vec z=(z_1,\ldots,z_n)$. Define
\begin{align}
\label{difference}
\nonumber
\Delta_{h_i\vec e_i}\,\phi(\vec z)=\frac{\phi(\vec z+h_i\vec e_i)-\phi(\vec z)}{h_i}
&=\phi(z_1,\ldots,z_{i-1},[z_i+h_i,z_i],z_{i+1},\ldots,z_n),\\
\Delta_{h_1\vec e_1,\ldots,h_r\vec e_r}\,\phi(\vec z)&=\Delta_{h_1\vec e_1}\ldots\Delta_{h_r\vec e_r}\,\phi(\vec z).
\end{align}
The latter equals $\Delta_{h_{\pi(1)}\vec e_{\pi(1)}}\ldots\Delta_{h_{\pi(r)}\vec e_{\pi(r)}}\,\phi(\vec z)$, for any permutation $\pi$ of $\{1,\ldots,r\}$.

\begin{lemma}
\label{boxs}
Let $1\leq r\leq n$ and $\vec\xi_i=h_i\vec e_{\pi(i)}$, $h_i\neq 0$, $1\leq i\leq r$, where $\pi$ is a function mapping $\{1,\ldots,r\}$ to $\{1,\ldots,n\}$. Then, for $\vec z\in\C^n$, $\phi\in \Acal(\C^n)$,
\begin{align}
\label{boxs1}
\Delta_{h_1\vec e_{\pi(1)},\ldots,h_r\vec e_{\pi(r)}}\phi(\vec z)=\int_{[0,1]^r}\frac{\partial^r \phi}{\partial z_{\pi(1)}\ldots\partial z_{\pi(r)}}\left(\vec z+\sum_{i=1}^r\vec\la(i)h_i\vec e_{\pi(i)}\right)d\vec\la,
\end{align}
where $\vec\la(i)$ denotes the $i$-th component of the vector $\vec\la$, and, hence,
\begin{align}
\label{boxs2}
\sup_{h_1,\ldots,h_r\in\C\setminus\{0\},\,\vec z\in\C^n}\big|\Delta_{h_1\vec e_{\pi(1)},\ldots,h_r\vec e_{\pi(r)}}\phi(\vec z)\big|
\leq\left\|\frac{\partial^r \phi}{\partial z_{\pi(1)}\ldots\partial z_{\pi(r)}}\right\|_\infty.
\end{align}
\end{lemma}

\begin{proof}
Note that \eqref{boxs2} is an immediate consequence of the representation \eqref{boxs1}.
The proof of \eqref{boxs1} goes by induction on $r$.
If $r=1$, then
\begin{align}
\label{d1}\Delta_{h_r\vec e_{\pi(r)}}\phi(\vec z)=\frac{1}{h_r}\int_0^1 \left(\frac{d}{d\la}\phi(\vec z+\la\,h_r\vec e_{\pi(r)})\right)d\la=\int_0^1 \frac{\partial \phi}{\partial z_{\pi(r)}}(\vec z+\la\,h_r\vec e_{\pi(r)})\,d\la.
\end{align}
Assume that the formula holds for $r-1$, that is,
\[\Delta_{h_1\vec e_{\pi(1)},\ldots,h_{r-1}\vec e_{\pi(r-1)}}\phi(\vec z)=\int_{[0,1]^{r-1}}\frac{\partial^{r-1}\phi}{\partial z_{\pi(1)}\ldots\partial z_{\pi(r-1)}}\left(\vec z+\sum_{i=1}^{r-1}\vec\mu(i)h_i\vec e_{\pi(i)}\right)d\vec\mu.\]
Combining the latter with \eqref{d1} gives
\begin{align*}
&\Delta_{h_1\vec e_{\pi(1)},\ldots,h_r\vec e_{\pi(r)}}\phi(\vec z)=\int_{[0,1]^{r-1}}\Delta_{h_r\vec e_{\pi(r)}}\left(\frac{\partial^{r-1}\phi}{\partial z_{\pi(1)}\ldots\partial z_{\pi(r-1)}}\right)
\left(\vec z+\sum_{i=1}^{r-1}\vec\mu(i)h_i\vec e_{\pi(i)}\right)d\vec\mu\\
&\quad=\int_{[0,1]^{r-1}}\int_0^1\frac{\partial^r\phi}{\partial z_{\pi(1)}\ldots\partial z_{\pi(r)}}\left(\vec z+\sum_{i=1}^{r-1}\vec\mu(i)h_i\vec e_{\pi(i)}+\la\,h_r\vec e_{\pi(r)}\right)d\la\,d\vec\mu=\eqref{boxs1}.
\end{align*}
\end{proof}

In the evaluation of G\^{a}teaux derivatives of operator functions, we will need the following representations for the difference and divided difference operators.

\begin{lemma}\label{lemma4.1}
Let $N\in\N$ and
\[f_N(z_1,\ldots,z_n)=\!\!\sum_{0\leq k_1,\ldots,k_n\leq N}\!\! c_{k_1,\ldots,k_n} z_1^{k_1}\ldots z_n^{k_n}.\] The following assertions hold for $h_i\neq 0, h_j\neq 0$.
\begin{enumerate}[(i)]
\item\label{4.1ii}
If $k_j\geq 2$, then
\begin{align*}
&f_N(z_1,\ldots,z_{j-1},[z_j+h_j,z_j+h_j,z_j],z_{j+1},\ldots,z_n)\\
&\quad\quad=\!\!\sum_{0\leq k_1,\ldots,k_n\leq N}\!\!c_{k_1,\ldots,k_n}\!\!\sum_{\substack{p_{0},\,p_{1},p_2\geq 0\\p_{0}+p_{1}+p_2=k_j-2}}\!\!
z_1^{k_1}\ldots z_{j-1}^{k_{j-1}}(z_j+h_j)^{p_0+p_2}z_j^{p_1}z_{j+1}^{k_{j+1}}\ldots z_n^{k_n}.
\end{align*}

\item\label{4.1i}
If $k_j\geq 1$, then
\begin{align*}
\Delta_{h_i\vec e_i,\,h_j\vec e_j}\,f_N(z_1,\ldots,z_n)
&=\!\!\sum_{0\leq k_1,\ldots,k_n\leq N}\!\!c_{k_1,\ldots,k_n}\!\!\sum_{\substack{q_{0},\,q_{1}\geq 0\\q_{0}+q_{1}=k_i-1}}\,
\sum_{\substack{p_{0},\,p_{1}\geq 0\\p_{0}+p_{1}=k_j-1}} \\
& z_1^{k_1}\ldots z_{i-1}^{k_{i-1}}(z_i+h_i)^{q_0}z_i^{q_1}z_{i+1}^{k_{i+1}}\ldots
z_{j-1}^{k_j-1}(z_j+h_j)^{p_0}z_j^{p_1}z_{j+1}^{k_{j+1}}\ldots z_n^{k_n}.
\end{align*}
\end{enumerate}
\end{lemma}

\begin{proof}
By linearity of the divided difference and operator defined in \eqref{difference}, it is enough to prove the lemma for
$p(z_1,\ldots,z_n)=z_1^{k_1}\ldots z_n^{k_n}$, where $k_1,\ldots,k_n\in\N\cup\{0\}$.
It is straightforward to see that
\[\Delta_{h_j\vec e_j}\,p(z_1,\ldots,z_n)=\sum_{\substack{p_{0},\,p_{1}\geq 0\\p_{0}+p_{1}=k_j-1}}\!\!
z_1^{k_1}\ldots z_{j-1}^{k_{j-1}}(z_j+h_j)^{p_0}z_j^{p_1}z_{j+1}^{k_{j+1}}\ldots z_n^{k_n}.\]
Subsequent evaluation of the divided difference gives \eqref{4.1ii} and of the operator \eqref{difference} gives \eqref{4.1i}.
\end{proof}

\begin{lemma}
\label{d2lemma}
Assume Notations \ref{not} and let $\Ical$ be a normed ideal. If 
$f\in \Acal(\D^n_{1+\eps})$, then $\ds t\mapsto\frac{d^2}{ds^2}\bigg|_{s=t}f(\Xt_n(s))$ exists in the operator norm. Moreover, for $f$ given by \eqref{sr},
\begin{align}
\label{d2fla}
\frac{d^2}{ds^2}\bigg|_{s=t}f(\Xt_n(s))=2\sum_{1\leq i< j\leq n}D^{i,j}_f(t)
+\sum_{1\leq j\leq n}D^{j,j}_f(t),
\end{align}
where
\begin{align*}
D^{i,j}_f(t)&=\!\!\sum_{k_1,\ldots,k_n\geq 0}c_{k_1,\ldots,k_n}
\bigg(T_{k_1,\ldots,k_{i-1}}(\Xt_{i-1}(t))\,\frac{d}{ds}\bigg|_{s=t}X_i(s)^{k_i}\,
T_{k_{i+1},\ldots,k_{j-1}}(_{i+1}\Xt_{j-1}(t))\\
&\quad\quad\frac{d}{ds}\bigg|_{s=t}X_j(s)^{k_j}\,
T_{k_{j+1},\ldots,k_n}(_{j+1}\Xt_n(t))\bigg),\quad\text{if } i<j,
\end{align*}
\begin{align*}
D^{j,j}_f(t)=\!\!\sum_{k_1,\ldots,k_n\geq 0}c_{k_1,\ldots,k_n}T_{k_1,\ldots,k_{j-1}}(\Xt_{j-1}(t))\,
\frac{d^2}{ds^2}\bigg|_{s=t}X_j(s)^{k_j}\,T_{k_{j+1},\ldots,k_n}(_{j+1}\Xt_n(t)).
\end{align*}
\begin{enumerate}[(i)]
\item
If $V_j\in\Ical$, then $\ds t\mapsto\frac{d^2}{ds^2}\bigg|_{s=t}f(\Xt_n(s))$ is continuous on $[0,1]$ in the norm $\|\cdot\|_\Ical$.
\item Assume that $\Ical^{1/2}$ is a normed ideal with ideal norm $\|\cdot\|_{\Ical^{1/2}}$ and that the inequality \eqref{normCS} holds. If $V_j\in\Ical^{1/2}$, $1\leq j\leq n$, then $\ds t\mapsto\frac{d^2}{ds^2}\bigg|_{s=t}f(\Xt_n(s))$ is continuous in $\|\cdot\|_\Ical$.
\end{enumerate}
\end{lemma}

\begin{proof}
Existence and continuity of $\frac{d^2}{ds^2}\big|_{s=t}f(\Xt_n(s))$ as well as the formula \eqref{d2fla} can be established completely analogously to the results of Lemma \ref{dlemma}.
\end{proof}

\begin{hyp}
\label{hyp2}
Let $\Ical$ be a normed ideal endowed with a positive and $\|\cdot\|_\Ical$-bounded trace $\tau_\Ical$. Assume that $\Ical^{1/2}$ is a normed ideal with ideal norm $\|\cdot\|_{\Ical^{1/2}}$ and that the inequality \eqref{normCS} holds.
Assume Notations \ref{not} and that $\At_n, \Bt_n\in\Ccal_n$ satisfy $V_j\in\Ical^{1/2}$, $1\leq j\leq n$.
\end{hyp}

\begin{thm}
\label{d2thm}
Assume Notations \ref{not} and Hypotheses \ref{hyp2}.
Suppose that there exists $t\in [0,1]$ such that $\Xt_n(t)\in\Ccal_n$ satisfies \eqref{dilfla}.
Then, for every $f\in \Acal(\D_{1+\eps}^n)$, $\eps>0$,
\begin{align}
\label{pd2est}
\left|\tau_\Ical\left(D^{i,j}_f(t)\right)\right|
\leq \big(\tau_\Ical(|V_i|^2)\big)^{1/2}\big(\tau_\Ical(|V_j|^2)\big)^{1/2} \left\|\frac{\partial^2{f}}{\partial z_i\partial z_j}\right\|_{L^\infty(\Omega^n)},
\end{align}
with $1\leq i\leq j\leq n$, where $\Omega=\T$. If, in addition, $\Xt_n(t)$ is a tuple of self-adjoint contractions, then \eqref{pd2est} holds with $\Omega=[-1,1]$ for $f\in \Acal((-1-\eps,1+\eps)^n)$, $\eps>0$.
\end{thm}

\begin{proof}
We provide the proof only in case $\Omega=\T$; the proof in case $\Omega=[-1,1]$ is a verbatim repetition.
Let $N\in\N$ and \[f_N(z_1,\ldots,z_n)=\!\!\sum_{0\leq k_1,\ldots,k_n\leq N}c_{k_1,\ldots,k_n}z_1^{k_1}\ldots z_n^{k_n}.\]
Since
\[\|D_f^{i,j}(t)\|_\Ical\leq \!\!\sum_{k_1,\ldots,k_n\geq 0}\!\! k_i\,k_j\,|c_{k_1,\ldots,k_n}| \,\|V_i\|_{\Ical^{1/2}}\|V_j\|_{\Ical^{1/2}},\] we have
\[\tau_\Ical\big(D^{i,j}_f(t)\big)=\lim_{N\rightarrow\infty}\tau_\Ical\big(D^{i,j}_{f_N}(t)\big),\]
so it is enough to prove \eqref{pd2est} for the function $f_N$. Applying \eqref{monder}, \eqref{monder2}, and pairwise commutativity of $X_1(t),\ldots,X_n(t)$ ensures
\begin{align*}
\tau_\Ical(D^{j,j}_{f_N}(t))
&=2\!\!\sum_{0\leq k_1,\ldots,k_n\leq N}\!\!\! c_{k_1,\ldots,k_n}\!\!\!\!\!\sum_{\substack{p_{0},\,p_{1},\,p_{2}\geq 0\\p_{0}+p_{1}+p_{2}=k_j-2}}\!\!\!
\tau_\Ical\big(T_{k_1,\ldots,k_{j-1},p_0+p_2,k_{j+1},\ldots,k_n}(\Xt_n(t))V_jX_j(t)^{p_1}V_j\big)
\end{align*}
and, if $i<j$,
\begin{align*}
\tau_\Ical\big(D^{i,j}_{f_N}(t)\big)
&=\!\!\!\sum_{0\leq k_1,\ldots,k_n\leq N}\!\! c_{k_1,\ldots,k_n}\!\!\!\sum_{\substack{q_{0},\,q_{1}\geq 0\\q_{0}+q_{1}=k_i-1}}\,
\sum_{\substack{p_{0},\,p_{1}\geq 0\\p_{0}+p_{1}=k_j-1}}\!\!\\
&\quad
\tau_\Ical\big(T_{k_1,\ldots,k_{i-1},q_0}(\Xt_i(t))
T_{p_1,k_{j+1},\ldots,k_n}(_{j}\Xt_n(t))V_i\,T_{q_1,k_{i+1},\ldots,k_{j-1},p_0}(_{i}\Xt_j(t))\,V_j\big).
\end{align*}

{\it Case 1:} Assume, in addition, that the tuple $\Xt_n(t)$ consists of normal contractions.

Let $E_{t,j}$ be the spectral measure of $X_j(t)$ and $E_t$ the joint spectral measure of the tuple $(X_1(t),\ldots,X_n(t))$. By the spectral theorem, we have
\begin{align*}
X_j(t)^{p_1}&=\int_\C z^{p_1}\,dE_{t,j}(z),\\
T_{k_1,\ldots,k_{j-1},p_0+p_2,k_{j+1},\ldots,k_n}(\Xt_n(t))
&=\int_{\C^n}z_1^{k_1}\ldots z_{j-1}^{k_{j-1}}z_j^{p_0+p_2}z_{j+1}^{k_{j+1}}\ldots z_n^{k_n}\,
dE_t(z_1,\ldots,z_n).
\end{align*}
Hence, for every $1\leq l\leq n$, there is a sequence of Borel partitions $(\delta_{m,l,\beta_l})_{1\leq \beta_l\leq m}$ of $\C$, a sequence of tuples of complex numbers $(z_{m,l,\beta_l})_{1\leq \beta_l\leq m}$, and also sequence of partitions $(\tilde\delta_{m,\beta})_{1\leq \beta\leq m}$ of $\C$ and a sequence of tuples of complex numbers $(\tilde z_{m,\beta})_{1\leq \beta\leq m}$ such that
\begin{align}
\label{tmonj}
\nonumber
&\tau_\Ical\big(T_{k_1,\ldots,k_{j-1},p_0+p_2,k_{j+1},\ldots,k_n}(\Xt_n(t))\,V_j X_j(t)^{p_1}V_j\big)\\
\nonumber
&=\lim_{m\rightarrow\infty}\!\!\sum_{1\leq \beta_1,\ldots,\beta_n,\beta\leq m}\!\! \bigg(z_{m,1,\beta_1}^{k_1}\ldots z_{m,j-1,\beta_{j-1}}^{k_{j-1}}z_{m,j,\beta_j}^{p_0+p_2}
z_{m,j+1,\beta_{j+1}}^{k_{j+1}}\ldots z_{m,n,\beta_n}^{k_n}\tilde z_{m,\beta}^{p_1}\\
&\quad\quad \tau_\Ical\big(E_t\big(\delta_{m,1,\beta_1}\times\ldots\times\delta_{m,n,{\beta_n}}\big)V_j E_{t,j}\big(\tilde\delta_{m,\beta}\big)V_j\big)\bigg).
\end{align}
From \eqref{tmonj} and Lemma \ref{ddder}, we derive
\begin{align}
\label{djj}
&\tau_\Ical(D^{j,j}_{f_N}(t))=2\lim_{m\rightarrow\infty}\sum_{1\leq \beta_1,\ldots,\beta_n,\beta\leq m}\\
\nonumber
&\bigg(f_N\big(z_{m,1,\beta_1},\ldots,z_{m,j-1,\beta_{j-1}},[z_{m,j,\beta_j},z_{m,j,\beta_j},\tilde z_{m,\beta}],z_{m,j+1,\beta_{j+1}},\ldots,z_{m,n,\beta_n}\big)\\
\nonumber
&\quad\quad\tau_\Ical\big(E_t\big(\delta_{m,1,\beta_1}\times\ldots\times\delta_{m,n,{\beta_n}}\big)
V_j E_{t,j}\big(\tilde\delta_{m,\beta}\big)V_j\big)\bigg).
\end{align}
Applying Lemmas \ref{ddder} and \ref{m2v}\eqref{m2vii} to \eqref{djj} provides the estimate \eqref{pd2est}.

Now we suppose that $i<j$ and denote by $E_{t,i,\ldots,j}$ the spectral measure defined by
$E_{t,i,\ldots,j}(S_i,\ldots,S_j)=E_t(\C\times\ldots\times\C\times S_i\times\ldots\times S_j\times\C\times\ldots\times\C)$. We have
\begin{align*}
T_{q_1,k_{i+1},\ldots,k_{j-1},p_0}(_{i}\Xt_j(t))=\int_{\C^{j-i+1}}z_i^{q_1}z_{i+1}^{k_{i+1}}\ldots z_{j-1}^{k_{j-1}}z_j^{p_0}\,dE_{t,i,\ldots,j}(z_i,\ldots,z_j).
\end{align*}
Hence, for every integer $i\leq \gamma\leq j$, there is a sequence of Borel partitions $(\tilde\delta_{m,\gamma,\alpha_\gamma})_{1\leq \alpha_\gamma\leq m}$ of $\C$ and a sequence of complex numbers $(\tilde z_{m,\gamma,\alpha_\gamma})_{1\leq \alpha_\gamma\leq m}$ such that
\begin{align}
\label{tmon}
\nonumber
&\tau_\Ical\big(T_{k_1,\ldots,k_{i-1},q_0}(\Xt_j(t))T_{p_1,k_{j+1},\ldots,k_n}(_{j}\Xt_n(t))\,
V_i\,T_{q_1,k_{i+1},\ldots,k_{j-1},p_0}(_{i}\Xt_j(t))\,V_j\big)\\
\nonumber
&=\lim_{m\rightarrow\infty}\sum_{1\leq \beta_1,\ldots,\beta_n,\alpha_i,\ldots,\alpha_j\leq m}\!\! \bigg(z_{m,1,\beta_1}^{k_1}\ldots z_{m,i-1,\beta_{i-1}}^{k_{i-1}}
z_{m,i,\beta_i}^{q_0}z_{m,j,\beta_j}^{p_1}
z_{m,j+1,\beta_{j+1}}^{k_{j+1}}\ldots z_{m,n,\beta_n}^{k_n}\\
\nonumber
&\quad\quad\quad
\tilde z_{m,i,\alpha_i}^{q_1}\tilde z_{m,i+1,\alpha_{i+1}}^{k_{i+1}}\ldots \tilde z_{m,j-1,\alpha_{j-1}}^{k_{j-1}}\tilde z_{m,j,\alpha_j}^{p_0}\\
&\quad\quad\quad
\tau_\Ical\big(E_t\big(\delta_{m,1,\beta_1}\times\ldots\times\delta_{m,n,{\beta_n}}\big)V_i E_{t,i,\ldots,j}\big(\tilde\delta_{m,i,\alpha_i}\times\ldots\times\tilde\delta_{m,j,\alpha_j}\big)
V_j\big)\bigg).
\end{align}
From \eqref{tmon} and Lemma \ref{lemma4.1}\eqref{4.1i}, we derive
\begin{align*}
&\tau_\Ical(D^{i,j}_{f_N}(t))=\lim_{m\rightarrow\infty}\sum_{1\leq \beta_1,\ldots,\beta_n,\alpha_i,\ldots,\alpha_j\leq m}\bigg(\Delta_{(z_{m,i,\beta_i}-\tilde z_{m,i,\alpha_i})\vec e_i,\,(z_{m,j,\beta_j}-\tilde z_{m,j,\alpha_j})\vec e_j}\\
&\quad f_N\big(z_{m,1,\beta_1},\ldots,z_{m,i-1,\beta_{i-1}},\tilde z_{m,i,\alpha_i}\ldots,\tilde z_{m,j,\alpha_j},z_{m,j+1,\beta_{j+1}},
\ldots,z_{m,n,\beta_n}\big)\\
&\quad\tau_\Ical\big(E_t\big(\delta_{m,1,\beta_1}\times\ldots\times\delta_{m,n,{\beta_n}}\big)V_i E_{t,i,\ldots,j}\big(\tilde\delta_{m,i,\alpha_i}\times\ldots\times\tilde\delta_{m,j,\alpha_j}\big)V_j\big)
\bigg).
\end{align*}
Application of Lemmas \ref{boxs} and \ref{m2v}\eqref{m2vii} completes the proof in case $\Xt_n(t)$ is a tuple of normal contractions.

{\it Case 2:} $\Xt_n(t)$ is a tuple of contractions.

Since there exists a tuple of normal contractions $\Ut_n$ on a Hilbert space $\mathcal{K}\supset\mathcal{H}$ such that \eqref{dilfla} holds, we have
\begin{align*}
&\tau_\Ical\big(T_{k_1,\ldots,k_{i-1},q_0}(\Xt_i(t))
T_{p_1,k_{j+1},\ldots,k_n}(_{j}\Xt_n(t))V_i\,T_{q_1,k_{i+1},\ldots,k_{j-1},p_0}(_{i}\Xt_j(t))\,V_j\big)\\
&=\tau_{\tilde\Ical}\big(P_\mathcal{H}T_{k_1,\ldots,k_{i-1},q_0}(\Ut_i)
T_{p_1,k_{j+1},\ldots,k_n}(_{j}\Ut_n)V_iP_\mathcal{H}\,T_{q_1,k_{i+1},\ldots,k_{j-1},p_0}(_{i}\Ut_j)\,
V_jP_\mathcal{H}\big),
\end{align*}
for $\tau_{\tilde\Ical}$ given by Proposition \ref{dilation}. Hence, by Case 1,
\begin{align*}
\left|\tau_\Ical\left(D^{i,j}_f(t)\right)\right|
&\leq \big(\tau_{\tilde\Ical}(|V_iP_\mathcal{H}|^2)\big)^{1/2}
\big(\tau_{\tilde\Ical}(|V_jP_\mathcal{H}|^2)\big)^{1/2} \left\|\frac{\partial^2{f}}{\partial z_i\partial z_j}\right\|_{L^\infty(\T^n)}\\
&\leq\big(\tau_{\Ical}(|V_i|^2)\big)^{1/2}\big(\tau_{\Ical}(|V_j|^2)\big)^{1/2} \left\|\frac{\partial^2{f}}{\partial z_i\partial z_j}\right\|_{L^\infty(\T^n)}.
\end{align*}
Similarly, we derive the estimate for $\left|\tau_\Ical\left(D^{j,j}_f(t)\right)\right|$.
\end{proof}

Adjusting \cite[Lemma 5.6]{dsD} gives the following integral representation for the trace of the Taylor remainder.

\begin{lemma}
\label{remr2}
Assume Notations \ref{not} and Hypotheses \ref{hyp2}.
Then, for every $f\in \Acal(\D_{1+\eps}^n)$, $\eps>0$,
\begin{align*}
\tau_\Ical\left(f(\Bt_n)-f(\At_n)-\frac{d}{ds}\bigg|_{s=0}f(\Xt_n(s))\right)
=\int_0^1(1-t)\,\tau_\Ical\left(\frac{d^2}{ds^2}\bigg|_{s=t}f(\Xt_n(s))\right)dt.
\end{align*}
\end{lemma}

\begin{proof}
Consider the function
\begin{align*}
\Psi(t)=\frac{d}{ds}\bigg|_{s=t}f(\Xt_n(s))-\frac{d}{ds}\bigg|_{s=0}f(\Xt_n(s)).
\end{align*}
By Lemma \ref{dlemma}, the above derivatives exist in the operator norm and the function $\Psi(t)$ is continuous in the operator norm. Hence, for every continuous linear functional $\phi$ on the space of bounded operators,
\begin{align*}
&\phi\left(f(\Bt_n)-f(\At_n)-\frac{d}{ds}\bigg|_{s=0}f(\Xt_n(s))\right)
=\phi(f(\Bt_n))-\phi(f(\At_n))-\frac{d}{ds}\bigg|_{s=0}\phi(f(\Xt_n(s)))\\
&=\int_0^1 \frac{d}{ds}\bigg|_{s=t}\phi\big(f(\Xt_n(s))\big)\,dt-
\frac{d}{ds}\bigg|_{s=0}\phi(f(\Xt_n(s)))=\int_0^1\phi(\Psi(t))\,dt.
\end{align*}
Therefore,
\begin{align}
\label{tviapsi}
f(\Bt_n)-f(\At_n)-\frac{d}{ds}\bigg|_{s=0}f(\Xt_n(s))=\int_0^1 \Psi(t)\,dt,
\end{align}
where the integral converges in the operator norm.

We claim that $\Psi$ is continuous in the ideal norm $\|\cdot\|_\Ical$ on $[0,1]$, but provide details only  for continuity at $t=0$.
Applying \eqref{dfla} and \eqref{mondif} gives
\begin{align}
\label{Psi}
&\Psi(t)-\Psi(0)=\sum_{j=1}^n
\sum_{k_1,\ldots,k_n\geq 0}\!\! c_{k_1,\ldots,k_n}\!\!\!\sum_{\substack{p_{0}^j,\,p_{1}^j\geq 0\\p_{0}^j+p_{1}^j=k_j-1}}\!\!\!\big(g_{j,p_0^j,p_1^j}(t)-g_{j,p_0^j,p_1^j}(0)\big)
\\
\nonumber
&=t\sum_{j=1}^n
\sum_{k_1,\ldots,k_n\geq 0}\!\! c_{k_1,\ldots,k_n}\!\left(\sum_{\substack{p_{0}^j,\,p_{1}^j\geq 0\\p_{0}^j+p_{1}^j=k_j-1}}\!\!\!\bigg(\sum_{i=1}^{j-1}S_{i<j}+\sum_{i=j+1}^n S_{i>j}\bigg)
+\sum_{\substack{p_{0}^j,\,p_{1}^j,\,p_2^j\geq 0\\p_{0}^j+p_{1}^j+p_2^j=k_j-1}}S_j\right),
\end{align}
where
\begin{align*}
S_{i<j}&=T_{k_1,\ldots,k_{i-1}}(\At_{i-1})X_i(t)^{q_0^i}\,V_i\, A_i^{q_1^i} T_{k_{i+1},\ldots,k_{j-1},p_0^j}(_{i+1}\Xt_j(t))\,V_j\, T_{p_1^j,k_{j+1},\ldots,k_n}(_{j}\Xt_n(t)),\\
S_{i>j}&=T_{k_1,\ldots,k_{j-1},p_0^j}(\At_j)V_j T_{p_1^j,\ldots,k_{i-1}}(_{j}\At_{i-1})
X_i(t)^{q_0^i}\,V_i\, A_i^{q_1^i} T_{k_{i+1},\ldots,k_n}(_{i+1}\Xt_n(t)),\\
S_j&=T_{k_1,\ldots,k_{j-1}}(\At_{j-1})X_j(t)^{p_0^j}\,V_j\,A_j^{p_1^j}\,V_j\, T_{p_2^j,k_{j+1},\ldots,k_n}(_{j}\Xt_n(t))\\
&\quad\quad+T_{k_1,\ldots,k_{j-1},p_0^j}(\At_j)\,V_j\,X_j(t)^{p_1^j}\,V_j\,A_j^{p_2^j}
T_{k_{j+1},\ldots,k_n}(_{j+1}\Xt_n(t)).
\end{align*}
Thus, $\Psi$ is continuous in the ideal norm $\|\cdot\|_\Ical$ and, along with \eqref{tviapsi}, this implies
\begin{align}
\label{almost}
\tau_\Ical\left(f(\Bt_n)-f(\At_n)-\frac{d}{ds}\bigg|_{s=0}f(\Xt_n(s))\right)
=\int_0^1\tau_\Ical(\Psi(t))\,dt.
\end{align}
One can see from \eqref{Psi} and \eqref{d2fla} that
\begin{align*}
\frac{d}{ds}\bigg|_{s=t}\tau_\Ical(\Psi(s))=\tau_\Ical\left(\frac{d^2}{ds^2}\bigg|_{s=t}f(\Xt_n(s))\right). \end{align*}
Therefore, integrating by parts in \eqref{almost} completes the proof.
\end{proof}

\begin{thm}
\label{tr2thm}
Assume Notations \ref{not}, Hypotheses \ref{hyp2}, and that $[V_i,V_j]=0$, for $1\leq i\neq j\leq n$. If $\{\At_n,\Bt_n\}\in\Ncal_n$, then there exist finite measures $\nu_{ij}$, $1\leq i\leq j\leq n$, on $\Omega^n$ such that
\begin{align}
\label{nin}
\|\nu_{ij}\|\leq\frac12\big(\tau_\Ical(|V_i|^2)\big)^{1/2}\big(\tau_\Ical(|V_j|^2)\big)^{1/2},
\end{align}
\begin{align}
\label{nutau}
\nu_{ij}(\Omega^n)=\frac12\tau_\Ical(V_iV_j),
\end{align}
and
\begin{align}
\label{tf2}
\nonumber
&\tau_\Ical\left(f(\Bt_n)-f(\At_n)-\frac{d}{ds}\bigg|_{s=0}f(\Xt_n(s))\right)
=\sum_{1\leq j\leq n}\int_{\Omega^n}\frac{\partial^2 f }{\partial z_j^2}(z_1,\ldots,z_n)\,d\nu_{jj}(z_1,\ldots,z_n)\\
&\quad\quad+2\sum_{1\leq i< j\leq n}
\int_{\Omega^n}\frac{\partial^2 f }{\partial z_i\partial z_j}(z_1,\ldots,z_n)\,d\nu_{ij}(z_1,\ldots,z_n),
\end{align}
for every $f\in\AD$, where $\Omega=\T$.

For every $1\leq j\leq n$ and $g\in\Acal(\D_{1+\eps})$,
\begin{align}
\label{connection2}
\int_{\Omega^n}g(z_j)\,d\nu_{jj}(z_1,\ldots,z_n)=\int_{\Omega}g(z)\,d\nu_{A_j,B_j}(z),
\end{align}
where $\nu_{A_j,B_j}$ is the measure given by Theorem \ref{dsmt2} for the pair $(A_j,B_j)$.

If, in addition, $\At_n$ and $\Bt_n$ are tuples of self-adjoint contractions, then there exist real-valued measures $\nu_{ij}$, $1\leq i\leq j\leq n$, such that  \eqref{nin}--\eqref{tf2} hold with $\Omega=[-1,1]$ for $f(z_1,\ldots,z_n)\in \Acal((-1-\eps,1+\eps)^n)$, $\eps>0$.
\end{thm}

\begin{proof}
Combining the results of Lemmas \ref{remr2} and \ref{d2lemma}, we obtain
\begin{align}
\label{tinside2}
\nonumber
&\tau_\Ical\left(f(\Bt_n)-f(\At_n)-\frac{d}{ds}\bigg|_{s=0}f(\Xt_n(s))\right)\\
&\quad=2\sum_{1\leq i< j\leq n}\int_0^1(1-t)\tau_\Ical(D^{i,j}_f(t))\,dt
+\sum_{1\leq j\leq n}\int_0^1(1-t)\tau_\Ical(D^{j,j}_f(t))\,dt.
\end{align}
By the Riesz-Markov representation theorem for a bounded linear functional on the space of continuous functions on a compact set and by the Hahn-Banach theorem, we deduce from \eqref{tinside2} and Theorem \ref{d2thm} existence of measures $\mu_j$ satisfying \eqref{nin} and \eqref{tf2}. Applying \eqref{tf2} to $f(z_1,\ldots,z_n)=z_i z_j$, $1\leq i\leq j\leq n$, gives \eqref{nutau} and applying \eqref{tf2} to $f(z_1,\ldots,z_n)=f(z_j)$ such that $\frac{\partial^2 f}{\partial z_j^2}=g$ gives \eqref{connection2}.
\end{proof}

\begin{cor}
\label{arem2}
Let $(A_1,A_2),(B_1,B_2)$ be tuples of commuting self-adjoint operators such that $\|A_1+\i A_2\|\leq 1$, $\|B_1+\i B_2\|\leq 1$, $V_1=B_1-A_1, V_2=B_2-A_2\in\Ical^{1/2}$, and $[V_1,V_2]=0$. Then, there exist real-valued measures $\nu_{11},\nu_{12},\nu_{22}$ satisfying \eqref{nin}--\eqref{connection2} with $\Omega=[-1,1]$ and supported in $\overline\D$.
\begin{enumerate}[(i)]

\item \label{arem2ii} For $f\in\Acal(\D_{1+\eps})$,
\begin{align*}
&\tau_\Ical\left(f(B_1+\i B_2)-f(A_1+\i A_2)-\frac{d}{ds}\bigg|_{s=0}f(A_1+sV_1+\i(A_2+sV_2))\right)\\
&\quad=\int_{\overline\D}f''(z)\,d(\nu_{11}-\nu_{22}+2\i\nu_{12})(z).
\end{align*}

\item \label{arem2iii}
Let $j=1,2$.
If $A_i=B_i$ for $i\in\{1,2\}\setminus\{j\}$, then $\nu_{ii}=\nu_{ij}=0$ and, for $g\in\Acal((-1-\eps,1+\eps))$, $\eps>0$,
\begin{align}
\label{con2}
\int_{\overline\D}g(x_i)\,d\nu_{jj}(x_1,x_2)=\frac12\tau\big(g(A_i)V_j^2\big).
\end{align}
\end{enumerate}
\end{cor}

\begin{proof}
The property \eqref{arem2ii} is straightforward. The representation \eqref{con2} follows from \eqref{tf2} applied to $f(x_1,x_2)=g(x_i)x_j^2$, $1\leq i\neq j\leq 2$.
\end{proof}

\begin{prop}
\label{Dix2thm}
Suppose $\tau_\Ical(\Ical^{3/2})=\{0\}$ and assume Notations \ref{not} and Hypotheses \ref{hyp2}. If $\At_n$ satisfies \eqref{dilfla}, then there exist finite measures $\nu_{ij}$ on $\T^n$, $1\leq i,j\leq n$, such that \eqref{nin}--\eqref{tf2} hold.
\end{prop}

\begin{proof}
The proof is analogous to the one of Proposition \ref{Dix2thm}. Firstly one establishes
\begin{align*}
\tau_\Ical\left(f(\Bt_n)-f(\At_n)-\frac{d}{ds}\bigg|_{s=0}f(\Xt_n(s))\right)
=\frac12\tau_\Ical\left(\frac{d^2}{ds^2}\bigg|_{s=0}f(\Xt_n(s))\right).
\end{align*}
Repeating the reasoning in the proof of Theorem \ref{tr2thm} completes the proof of existence of measures satisfying \eqref{nin}--\eqref{tf2}.
\end{proof}

\section*{Acknowledgment}

The author is thankful to Mike Jury for consultation on multivariate unitary dilations.

\bibliographystyle{plain}

\end{document}